\documentclass[12pt, leqno]{article}

\usepackage{amssymb,amsfonts,amsthm}
\usepackage{graphicx}
\usepackage{subfigure}
\usepackage{amssymb}
\usepackage{amsthm}
\usepackage{amsmath}
\usepackage{bm}             
\usepackage[mathscr]{eucal}
\usepackage{color}
\usepackage{cancel}
\usepackage{enumerate}
\usepackage{psfrag}
\usepackage{appendix}

\newcommand{\R}{\mathbb{R}}

\newcommand{\eps}{\epsilon}

\newcommand{\mcR}{\mathcal{R}}
\newcommand{\mfR}{\mathfrak{R}}
\newcommand{\mcW}{\mathcal{W}}
\newcommand{\mcL}{\mathcal{L}}

\newtheorem{theorem}{Theorem}

\newtheorem{lemma}[theorem]{Lemma}

\theoremstyle{definition}

\usepackage{fullpage}

\title{\Huge{Arbitrarily large solutions of the Vlasov-Poisson system}}
\author{Jonathan Ben-Artzi\\ {\small School of Mathematics}\\ {\small Cardiff University}\\ {\small Cardiff, United Kingdom}\\ {\small \tt Ben-ArtziJ@cardiff.ac.uk} \\[1cm]
Simone Calogero\\{\small Department of Mathematics}\\ {\small Chalmers Institute of Technology, University of Gothenburg}\\ {\small Gothenburg, Sweden}\\ {\small \tt calogero@chalmers.se}\\[1cm]
Stephen Pankavich \\ {\small Department of Applied Mathematics and Statistics}\\ {\small Colorado School of Mines}\\ {\small Golden, CO USA}\\ {\small \tt pankavic@mines.edu} }
\date{\today}

\begin{document}
\maketitle

\begin{abstract}
We study smooth, global-in-time solutions of the Vlasov-Poisson system in the plasma physical case that possess arbitrarily large charge densities and electric fields. In particular, we construct two classes of solutions with this property. The first class are spherically-symmetric solutions that initially possess arbitrarily small density and field values, but attain arbitrarily large values of these quantities at some later time. Additionally, we construct a second class of spherically-symmetric solutions that possess any desired mass and attain arbitrarily large density and field values at any later prescribed time.
\end{abstract}

\section{Introduction}
\subsection{The main results}
In the classical limit (i.e., as the speed of light $c \to \infty$) the motion of a monocharged, collisionless plasma is given by the Vlasov-Poisson (VP) system:
 \begin{equation}\label{VP}
\left \{ \begin{aligned}
& \partial_{t}f+v\cdot\nabla_{x}f+E \cdot\nabla_{v}f=0\\
& \rho(t,x)=\int_{\mathbb{R}^3} f(t,x,v)\,dv\\
& E(t,x) = \int_{\mathbb{R}^3} \frac{x-y}{\vert x - y \vert^3} \rho(t,y) \ dy.
\end{aligned} \right .
\end{equation}
Here, $t \geq 0$ represents time, $x \in \mathbb{R}^3$ is position, and $v \in \mathbb{R}^3$ represents momentum.  Additionally, $f(t,x,v) \geq 0$ is the particle density, $\rho(t,x)$ is the associated charge density, $E(t,x)$ is the self-consistent electric field generated by the charged particles, and we have chosen units such that the mass and charge of each particle are normalized to one.
In the present paper, we consider the Cauchy problem and therefore require given initial data
$$f(0,x,v) = f_0(x,v) \geq 0$$
to complete the description of the system. 
We refer to \cite{Glassey} as a general reference to provide background for this well-known plasma model, but one important property that will be utilized throughout this paper is the {\it a priori} conservation of total mass of the system, namely
$$M = \iint_{\mathbb{R}^6} f(t,x,v) \ dv dx = \iint_{\mathbb{R}^6} f_0(x,v) \ dv dx.$$

In this paper we prove that over intermediate timescales, solutions to \eqref{VP} can give rise to charge densities and electric fields that become arbitrarily large. More specifically, our first main result shows that one may construct solutions of \eqref{VP} whose density and field are initially as small as desired, but which become arbitrarily large at some later time.
\begin{theorem}
\label{T1}
For any constants $C_1, C_2 > 0$ there exists a smooth, spherically symmetric solution of the Vlasov-Poisson system such that
$$ \Vert \rho(0) \Vert_\infty, \quad \Vert E(0) \Vert_\infty  \leq C_1$$
but for some time $T > 0$,
$$ \Vert \rho(T) \Vert_\infty, \quad \Vert E(T) \Vert_\infty  \geq C_2.$$
\end{theorem}
This theorem is inspired by a similar result obtained by Rein \& Taegert \cite{RT} for spherically symmetric solutions of the gravitational Vlasov-Poisson system. We note, however, that the convex, rather than concave, nature of the spatial characteristics in the plasma case gives rise to drastically different particle behavior, and therefore we must use new tools and a different argument within the proof. 

The next main result removes the condition on initial data and shows that one may construct solutions possessing any desired mass and whose density and field are arbitrarily large at any given time.
\begin{theorem}
\label{T2}
For any constants $C_1, C_2> 0$ and any $T > 0$ there exists a smooth, spherically symmetric solution of the Vlasov-Poisson system such that
$$ M = \iint_{\mathbb{R}^6} f_0(x,v) \ dv dx = C_1$$
and
$$ \Vert \rho(T) \Vert_\infty, \quad  \Vert E(T) \Vert_\infty  \geq C_2.$$
\end{theorem}
This result complements H\"{o}rst's decay theorem \cite{Horst}, which states that the density and field generated by any spherically-symmetric solution of \eqref{VP} must obey sharp asymptotic decay estimates (see discussion below). Our result, on the other hand, demonstrates that the time needed for solutions to transition from their intermediate asymptotic behavior, during which they may attain large values, to their final asymptotic behavior can be made arbitrarily large, even if the total mass is taken to be small.
Finally, we remark that because we construct spherically symmetric solutions of \eqref{VP}, they are also solutions of the nonrelativistic Vlasov-Maxwell system. Hence, Theorems \ref{T1} and \ref{T2} display the intermediate behavior of solutions to this system as well.

\subsection{Background and previous results}
It is well-known that given smooth initial data, the Vlasov-Poisson system \eqref{VP} possesses a smooth global-in-time solution \cite{LP, Pfaf}.  A remaining open question concerns the large-time asymptotic behavior of the system; more specifically, 
whether for $t>0$ sufficiently large there are $C, a, b > 0$ such that
$$\Vert \rho(t) \Vert_\infty \leq C(1+t)^{-a}, \qquad \Vert E(t) \Vert_\infty \leq C(1+t)^{-b}.$$
Of course, one would expect the repulsive nature of the electrostatic interaction to cause particles to separate rapidly, with the optimal rates of $a = 3$ and $b=2$ resulting from velocity averaging and particle dispersion, but a proof of such a result has remained elusive. To date, the best decay estimate for \eqref{VP} occurs within \cite{Yang} and yields $b = \frac{1}{6}$ with no associated estimate for the charge density. 

In the case of spherically-symmetric initial data $f_0$, the solution $f(t)$ is known to remain spherically-symmetric and an affirmative answer to the asymptotic behavior question was provided by H\"{o}rst \cite{Horst}.  Within this paper, it was shown that both terms decay for large time with associated exponents $a = 3$ and $b = 2$.
Results regarding the large time behavior of solutions to the (repulsive) Vlasov-Poisson system exist in other special cases, including small data \cite{BD}, the problem posed on the spatial torus \cite{VM}, and in a one-dimensional setting \cite{BKR, GPS, GPS2, Sch}.
Additionally, Illner \& Rein proved in \cite{IR} that both the potential energy of the system \eqref{VP} and $\Vert \rho(t) \Vert_{5/3}$ decay to zero as $t \to \infty$. In the attractive (gravitational) case, the asymptotic structure of solutions is much more complicated and partial results have been provided in~\cite{C} and \cite{RT}.

As mentioned above, this paper was inspired by \cite{RT}. There, the authors establish a result  similar to our Theorem 1 for the gravitational case, i.e. where the force is attractive. Their proof relies on a careful analysis of individual particle trajectories, which are all concave (intuitively they should all collapse toward the origin, though this is not achieved in finite time). They compare these trajectories with trajectories of an explicit spatially-homogeneous solution, and show both solutions possess a common core where they agree.



\subsection{Spherically symmetric coordinates}
Since we will be working with spherically-symmetric solutions, it will be useful to consider new variables that completely describe solutions with such symmetry.  In particular, defining the spatial radius, inward velocity, and square of the angular momentum by
\begin{equation}
\label{ang}
r = \vert x \vert, \qquad w = \frac{x \cdot v}{r}, \qquad \ell = \vert x \times v \vert^2,
\end{equation}
the spherical-symmetry of $f_0$ implies that the distribution function, charge density, and electric field take special forms. Namely, $f = f(t,r,w,\ell)$ satisfies the reduced Vlasov equation
\begin{equation}
\label{vlasovang}
\partial_{t}f+w\partial_r f+\left ( \frac{\ell}{r^3} + \frac{m(t,r)}{r^2} \right ) \partial_w f=0
\end{equation}
where
\begin{equation}
\label{massang}
m(t,r) = 4\pi \int_0^r s^2 \rho(t,s) \ ds
\end{equation}
and 
\begin{equation}
\label{rhoang}
\rho(t,r) = \frac{\pi}{r^2} \int_0^\infty \int_{-\infty}^{\infty} f(t,r,w,\ell) \ dw \ d\ell.
\end{equation}
The electric field is given by the expression
\begin{equation}
\label{fieldang}
E(t,x) = \frac{m(t,r)}{r^2} \frac{x}{r}.
\end{equation}
Whenever necessary, we will abuse notation so as to use both Cartesian and angular coordinates to refer to functions;  for instance the particle density $f$ will be written both as $f(t,x,v)$ and $f(t,r,w,\ell)$.

In the angular coordinates described above, the characteristics of the Vlasov equation also assume a reduced form, namely
\begin{equation}
\label{charang}
\left \{
\begin{aligned}
&\frac{d}{ds}\mcR(s)=\mcW(s),\\
&\frac{d}{ds}\mcW(s)= \frac{\mcL(s)}{R(s)^3} + \frac{m(s, \mcR(s))}{\mcR(s)^2},\\
&\frac{d}{ds}\mcL(s)= 0.
\end{aligned}
\right.
\end{equation}
We will study forward characteristics of the system with initial conditions
$$\mcR(0) = r, \qquad \mcW(0) = w, \qquad \mcL(0) = \ell$$
and note that the traditional convention for notational has been shortened so that
$$\mcR(s) = \mcR(s,0,r,w,\ell), \qquad \mcW(s) = \mcW(s,0,r,w,\ell), \qquad \mcL(s) = \mcL(s,0,r,w,\ell).$$
In particular, because the angular momentum of particles is conserved in time on the support of $f(t)$, we note that $\mcL(s) = \ell$ for every $s \geq 0$.
Throughout, we will estimate particle behavior on the support of $f$, thus for convenience we define for all $t \geq 0$
$$S(t) = \{ (r,w,\ell) : f(t,r,w,\ell) > 0\}$$
so that, in particular
$$S(0) = \{ (r,w,\ell) : f_0(r,w,\ell) > 0\}.$$

\subsection{Paper organization}
The proofs of Theorems~\ref{T1} and \ref{T2} are contained within Section~\ref{proofs}, while Section~\ref{lemmas} is devoted to proving some technical lemmas. 
Additionally, we remark that a theorem similar to these, but allowing for a given initial kinetic energy of any size while generating an arbitrarily large charge density and electric field at some later time, can also be established using our methods. This will be clearer from the proofs of these results, and the implied relationship between the particle positions and momenta on $S(0)$, and the time $T>0$.

\section{Proof of the main results}
\label{proofs}

\subsection{Class of initial data}
We begin by defining two classes of functions --  $\mathfrak{J}$ and $\mathfrak{K}$ -- from which initial data will be chosen to prove Theorems 1 and 2, respectively.

Given $a_0>0$, $a_1<0$ and $\epsilon>0$, we define the class $\mathfrak{J}(a_0,a_1,\epsilon)$ of initial data for the Vlasov-Poisson system to consist of the functions $f_0\in C^1_c(\R^6; [0,\infty))$, such that
\begin{enumerate}
\item The initial distribution $f_0$ is spherically symmetric. In particular, $f_0=f_0(r,w,\ell)$;
\item For every $(r,w,\ell) \in S(0)$, we have
\begin{equation}
\label{suppcond}
\left (r + \frac{a_0}{\vert a_1 \vert}w \right )^2 + \ell r^{-2} \left (\frac{a_0}{a_1} \right)^2 < \frac{\eps^2}{a_1^2}
\end{equation}
and
\begin{equation}
\label{suppcond2}
a_0-\delta_r< r<a_0+\delta_r \quad \mathrm{with}  \quad \delta_r=\epsilon^3;
\end{equation}
\item The initial charge density $\rho_0=\int f_0\,dv$ satisfies
\begin{equation}
\label{boundrho}
\rho_0(r)\leq \frac{3}{4\pi a_0^3},\qquad\forall r > 0,
\end{equation}
and
\begin{equation}
\label{equalrho}
\rho_0(r) = \frac{3}{4\pi a_0^3},\qquad \text{for }r \in \left [a_0 - \frac{1}{2}\delta_r, a_0 + \frac{1}{2} \delta_r \right ].
\end{equation}

\end{enumerate}

Next, given $a_0>0$, $a_1<0$, $\epsilon>0$ and $M>0$, let the class $\mathfrak{K}(a_0,a_1,\epsilon,M)$ of initial data for the Vlasov-Poisson system consist of the functions $f_0\in C^1_c(\R^6; [0,\infty))$ that satisfy Conditions 1 and 2 above, and in addition, possess total mass equal to $M$:
\[
\iint_{\R^6}f_0(x,v)\,dv\,dx=M.
\]
The solutions of Theorem 1 will be constructed by choosing data in the class $\mathfrak{J}$, while the solutions in Theorem 2 are launched by data in the class $\mathfrak{K}$.
We note that \eqref{suppcond} enforces the construction of data that is arbitrarily close to a particle distribution that gives rise to a spatially-homogeneous (i.e., $\rho_0(r)$ is independent of $r$) solution of \eqref{VP}.  However, \eqref{suppcond2} imposes that the spatial support of the data be contained within a spherical shell with radius centered about $a_0$. In the proof of Theorems 1 and 2, this is essential as it ensures that particles may not approach the origin too quickly, which would cause them to disperse and decrease their density. Finally, \eqref{boundrho} guarantees that the density of the data on this spherical shell is bounded above by data that launches an associated spatially-homogeneous solution.

Notice that~\eqref{suppcond} further implies
\begin{equation}
\label{suppcondrw}
\left \vert r + \frac{a_0}{\vert a_1 \vert}w \right \vert < \frac{\eps}{|a_1|}
\end{equation}
and
\begin{equation}
\label{suppcondl}
\ell  < \left (\frac{r}{a_0} \right )^2 \eps^2
\end{equation}
on $S(0)$ for $f_0\in\mathfrak{J}\cup\mathfrak{K}$.
Additionally, as $a_1<0$, the support condition~\eqref{suppcond} further implies $w \in (a_1 - \delta_w, a_1 + \delta_w)$ on $S(0)$ for $f_0\in\mathfrak{J}\cup\mathfrak{K}$, where 
\begin{equation}\label{deltaw}
\delta_w = \frac{\vert a_1 \vert \delta_r + \eps}{a_0}.
\end{equation}
Hence, all particles possess an initial inward velocity belonging to this interval.

To validate our choice of initial data, we show that $\mathfrak{J}$ and $\mathfrak{K}$ are not empty. Following~\cite{RT},
let $H:[0,\infty) \to [0,\infty)$ be any function satisfying
$$\int_{\mathbb{R}^3} H(\vert u \vert^2) \ du = \frac{3}{4\pi}$$
with $supp(H) \subset [0,1]$.
We rescale this function for any $\eps > 0$ by defining
$$H_\eps(\vert u \vert^2) = \frac{1}{\eps^3} H\left (\frac{\vert u \vert^2}{\eps^2} \right )$$
so that 
$$\int_{\mathbb{R}^3} H_\eps(\vert u \vert^2) \ du = \frac{3}{4\pi}$$
and
$supp(H_\eps) \subset [0,\eps^2].$
Further, for every $\eps > 0$, $x,v \in \mathbb{R}^3$, $a_0>0$, and $a_1<0$ define
$$h_\eps(x,v) = H_\eps(\vert a_1 x - a_0 v \vert^2).$$
It follows that
\[
\int_{\mathbb{R}^3} h_\eps(x,v) \ dv = \frac{3}{4\pi a_0^3}
\]
for every $x \in \mathbb{R}^3$. 
We also choose a cut-off function $\phi \in C^\infty\left ((0,\infty); [0,1]\right )$ satisfying
\[
\left \{
\begin{array}{ll}
& \phi(r) = 0 \quad \mathrm{for} \quad r \not\in [a_0 - \delta_r, a_0 + \delta_r],\\
& \phi(r) = 1 \quad \mathrm{for} \quad r \in (a_0 - \frac{1}{2}\delta_r, a_0 + \frac{1}{2}\delta_r),\quad \mathrm{with} \quad \delta_r=\epsilon^3.
\end{array}
\right.
\]
Then, we claim that
$$f_0(x,v) := h_\eps(x,v) \phi(\vert x \vert)\in\mathfrak{J},\quad \tilde{f}_0(x,v):=M\frac{f_0(x,v)}{\|f_0\|_{L^1}}\in\mathfrak{K}.$$
Indeed, from the upper bound on the support of $H_\eps$, we have on the support of $f_0(x,v)$ the inequality
$$\vert a_1 x - a_0 v \vert^2 < \eps^2.$$
Using the angular coordinates of \eqref{ang} and dividing by $a_1^2 \neq 0$, this can be seen to be equivalent to~\eqref{suppcond}. It is straightforward to verify that $f_0$ and $\tilde{f}_0$ satisfy the remaining properties in the definitions of $\mathfrak{J}$ and $\mathfrak{K}$, as the conditions on $\phi$ imply 
$$ \int_{\mathbb{R}^3} f_0(x,v) \ dv = \left (\int_{\mathbb{R}^3} h_\eps(x,v) \ dv \right) \phi (|x|) = \frac{3}{4\pi a_0^3} \phi(|x|)$$ 
and thus \eqref{suppcond2}-\eqref{equalrho} hold. We also remark that since initial data in $\mathfrak{J}\cup\mathfrak{K}$ are spherically-symmetric, they must give rise to global-in-time, spherically-symmetric solutions of \eqref{VP}.

\subsection{Proof of Theorem~\ref{T1}}
In the first result we choose $f_0\in\mathfrak{J}(a_0,a_1,\epsilon)$. The parameter $a_0$ will be fixed and we may choose $\vert a_1 \vert$ sufficiently large, $\eps$ sufficiently small, and $T$ sufficiently small so that particles are quickly concentrated near the origin and obtain radial positions as small as one desires, thereby causing the density and field to become arbitrarily large at time $T > 0$.

\begin{proof}
Let $C_1, C_2>0$ be given, and define the constant
$$a_0 = \left (\frac{32}{C_1} \right )^{1/3}.$$
Let $\eps > 0$ satisfy
$$\eps < \min\left \{1,\frac{1}{4}a_0, \frac{1}{200^3a_0 C_2} \right \}.$$ 
%
and set $$ a_1 = -\frac{1}{\eps^2} \qquad \mathrm{and} \qquad T = \frac{a_0}{\vert a_1 \vert} - 20\eps^4.$$
We note that the upper bounds on $\eps$ imply $T > 0$.
Along with the condition $\delta_r = \eps^3$, these choices imply (see \eqref{deltaw}) $$ \qquad \delta_w = \frac{2\eps}{a_0}.$$
With this, \eqref{boundrho} implies that the total mass obeys the following upper bound
\begin{eqnarray*}
M &=& \int_{\mathbb{R}^3} \rho_0(x) \ dx=4\pi \int_{a_0-\delta_r}^{a_0+\delta_r}\rho_0(r) r^2\,dr\\
& \leq & \frac{1}{a_0^3} \left [ (a_0+\delta_r)^3 - (a_0 - \delta_r)^3 \right ]\\
& = & \frac{6\delta_r}{a_0} +  \frac{2\delta_r^3}{a_0^3} \leq \frac{8\delta_r}{a_0} = \frac{8\eps^3}{a_0},
\end{eqnarray*}
%
while \eqref{equalrho} implies that $M$ has the following lower bound
\begin{eqnarray*}
M &\geq& 4\pi \int_{a_0-\frac{1}{2}\delta_r}^{a_0+\frac{1}{2}\delta_r}\rho_0(r) r^2\,dr\\
& \geq & \frac{1}{a_0^3} \left [ (a_0+\frac{1}{2}\delta_r)^3 - (a_0 - \frac{1}{2}\delta_r)^3 \right ]\\
& = & \frac{3\delta_r}{a_0} +  \frac{\delta_r^3}{4a_0^3} \geq  \frac{3\delta_r}{a_0} = \frac{3\eps^3}{a_0}.
\end{eqnarray*}
Thus, we find
\begin{equation}
\label{mass}
3a_0^{-1}\eps^3 \leq M \leq 8a_0^{-1} \eps^3
\end{equation}
and, in particular, this implies $M\leq 8a_0^{-1}$.
On $S(0)$, the upper bounds on $\eps$ further imply
\begin{equation}
\label{rw}
\left \{
\begin{gathered}
\frac{1}{2}a_0 < a_0 - \delta_r < r  < a_0 + \delta_r < \frac{3}{2}a_0\\
-\frac{3}{2}\eps^{-2} < a_1 - \delta_w < w < a_1 + \delta_w < -\frac{1}{2} \eps^{-2}.
\end{gathered}
\right.
\end{equation}
Additionally, \eqref{rw} combined with \eqref{suppcondl} implies a uniform upper bound on the angular momentum on $S(0)$, namely
\begin{equation}
\label{l}
\ell < \left (\frac{3}{2} \right )^2 \eps^2 \leq 3.
\end{equation}

To prove the conclusions of the theorem at time zero, we first notice that by~\eqref{boundrho}
$$\Vert \rho(0) \Vert_\infty \leq \frac{3}{4\pi a_0^3} \leq C_1.$$
Similarly, due to \eqref{fieldang} and \eqref{suppcond2} the field satisfies $|E(0,x)| = 0$ for $ \vert x \vert < a_0 - \delta_r$, while for $\vert x \vert > a_0 + \delta_r$
$$\vert E(0,x) \vert \leq \frac{M}{r^2} \leq \frac{M}{a_0^2} \leq \frac{8}{a_0^3}.$$
Finally, for $a_0 - \delta_r \leq \vert x \vert \leq a_0 + \delta_r$, we have
$$\vert E(0,x) \vert \leq \frac{M}{r^2} \leq \frac{M}{\left(\frac{1}{2} a_0 \right)^2} = \frac{4M}{a_0^2} \leq \frac{32}{a_0^3}.$$
Hence, we find
$$\Vert E(0) \Vert_\infty \leq \frac{32}{a_0^3} \leq C_1.$$
Therefore, we merely need to establish the contrasting inequalities at time $T$ to complete the proof.

Since the trajectories of particle positions are convex, they must each attain a minimum, and we use this construction to create a uniform lower bound over $S(0)$ on the time until particles attain their minima. Because the enclosed mass satisfies $$0 \leq m(t,r) \leq M \leq 8a_0^{-1}$$ for all $t, r \geq 0$, we see from \eqref{charang} that $\mcR(t)$ satisfies
$$0 \leq \ddot{\mcR}(t) - \ell\mcR(t)^{-3} \leq 8a_0^{-1}\mcR(t)^{-2}$$
with $\mcR(0) = r > 0$ and $\dot{\mcR}(0) = w < 0$. 
In order to exclude those particles in $S(0)$ with vanishing angular momentum, we define  $$S_+ = \{ (r, w, \ell) \in S(0): \ell > 0\}.$$
Using Lemma \ref{L1} with $L = \ell$ and $P = 8a_0^{-1}$, we find for each $(r,w,\ell) \in S_+$ a time $T_0(r,w,\ell)$ such that
$$T_0 \geq \frac{r}{\vert w \vert} \left ( 1 - \sqrt{\frac{\ell+8a_0^{-1}r}{r^2w^2 + \ell + 8a_0^{-1}r}} \right )$$
and
$$\dot{\mcR}(t) \leq 0 \quad \mathrm{for} \quad  t \in [0,T_0].$$
Estimating on $S_+$, we use \eqref{rw}, \eqref{l}, and the basic inequality
$$ \frac{1}{A + x} \geq \frac{1}{A} - \frac{x}{A^2}$$
for any $x,A > 0$, in order to arrive at
\begin{eqnarray*}
T_0 & > & \frac{r}{\vert w \vert} \left ( 1 - \frac{\sqrt{\ell + 8a_0^{-1}r}}{r \vert w \vert} \right )\\
& = & \frac{r}{\vert w \vert} - \frac{\sqrt{\ell + 8a_0^{-1}r}}{w^2}\\
& \geq & \frac{a_0 -\delta_r}{\vert a_1 \vert + \delta_w} - \frac{\sqrt{3 + 8a_0^{-1}(\frac{3}{2}a_0)}}{\left ( \frac{1}{2} \eps^{-2} \right )^2}\\
& \geq & (a_0 -\delta_r)\left (\frac{1}{\vert a_1 \vert} -  \frac{\delta_w}{a_1^2} \right) - 16\eps^{4}\\
& \geq & \frac{a_0}{\vert a_1 \vert} - \frac{\delta_r}{\vert a_1 \vert} - \frac{a_0 \delta_w}{a_1^2} - 16\eps^4\\
& \geq & \frac{a_0}{\vert a_1 \vert} - \left (\eps^5 + 2\eps^{5} + 16\eps^4\right )\\
& \geq & \frac{a_0}{\vert a_1 \vert} - 20\eps^4 = T.
\end{eqnarray*}

Therefore, $T \in [0, T_0)$ for every $(r,w,\ell) \in S_+$ and we apply Lemma \ref{L2} to find
$$\mcR(T)^2 \leq (r + wT)^2 + (\ell r^{-2} + 8a_0^{-1}r^{-1})T^2.$$
Because $T = \frac{a_0}{\vert a_1 \vert} - 20\eps^4$ this further implies
\begin{eqnarray*}
\mcR(T)^2 & \leq & \left (r + \frac{a_0}{|a_1|}w - 20\eps^4 w \right)^2 + \ell r^{-2} \left ( \frac{a_0}{|a_1|} - 20\eps^4 \right)^2 + 8a_0^{-1}r^{-1}T^2\\
& \leq & \left (r + \frac{a_0}{|a_1|}w \right)^2  +  \ell r^{-2}\left (\frac{a_0}{a_1} \right)^2 + 2|w|\left (20\eps^4 \right ) \left \vert r + \frac{a_0}{|a_1|}w \right \vert \\
& \ & + w^2\left (20\eps^4 \right )^2 
+ \ell r^{-2}\left (20\eps^4 \right )^2 + 8a_0^{-1}r^{-1}T^2.
\end{eqnarray*}

Using the conditions on $(r,w,\ell) \in S(0)$, namely \eqref{suppcond}, \eqref{suppcondrw}, \eqref{rw}, and \eqref{l} this yields 
\begin{eqnarray*}
\mcR(T)^2 & \leq & \frac{\eps^2}{a_1^2} + 2\left ( \frac{3}{2} \eps^{-2} \right )20\eps^4 \frac{\eps}{\vert a_1 \vert} + \left ( \frac{3}{2} \eps^{-2} \right )^2(20\eps^4)^2 \\
& \ &   + 3\left (\frac{1}{2}a_0 \right )^{-2}(20\eps^4)^2 + 8a_0^{-1}\left (\frac{1}{2}a_0 \right )^{-1}\left ( \frac{a_0}{a_1} \right)^2.\\
& \leq & \eps^6 + 60 \eps^5 + 900\eps^4 + \frac{4800}{a_0^2}\eps^{8} + 16\eps^4\\
& \leq & 10000\eps^4.
\end{eqnarray*}
Since this provides a uniform bound on $\mcR(T)$ over the set $S_+$, we take the supremum over all such triples to find
$$\sup_{(r,w,\ell) \in S(0)} \mcR(T, 0, r, w,\ell) = \sup_{(r,w,\ell) \in S_+} \mcR(T, 0, r, w,\ell) \leq 100 \eps^2.$$

Finally, invoking Lemma \ref{L3}, the upper bound on spatial characteristics implies a lower bound on the charge density and therefore using \eqref{mass}
$$ \Vert \rho(T) \Vert_\infty \geq \frac{3M}{4\pi \left ( 100\eps^2 \right)^3 } \geq \frac{1}{200^3a_0\eps^3} \geq C_2.$$
The same lemma also provides a lower bound on the field so that 
$$\Vert E(T) \Vert_\infty \geq \frac{M}{(100\eps^2)^2} \geq \frac{3}{100^2a_0 \eps} \geq C_2,$$
and the proof is complete.
\end{proof}

\subsection{Proof of Theorem~\ref{T2}}
For the second result we choose $f_0\in\mathfrak{K}(a_0,a_1,\epsilon,M)$. Unlike the first result, $T$ will be given here and we may choose $a_0$ and $\vert a_1 \vert$ sufficiently large so that particles are far enough from the origin that the initial large velocities they experience will concentrate them about the origin only near the given time $T$. As before, this behavior implies that the density and field become arbitrarily large at this time.

\begin{proof}
Let $C_1,C_2>0$ and $T > 0$ be given, and define the constant
$$C_0 = 3 + 12\sqrt{1+ C_1T} > 1.$$
Let $\eps > 0$ satisfy
$$\eps < \min\left \{1, \frac{T}{C_0}, \left (\frac{1}{(8C_0)^3}\frac{C_1}{6C_2} \right )^{1/2} \right\}$$ 
and set $$M=C_1,\qquad a_1 = -\eps^{-2}, \qquad  \eta = C_0\eps^3, \qquad \mathrm{and} \quad a_0 = \eps^{-2}(T + \eta)$$
so that $T = \frac{a_0}{\vert a_1 \vert} - \eta$.
These choices along with $\delta_r = \eps^3$ and~\eqref{deltaw} imply $$\delta_w = \frac{2\eps^3}{T + \eta} \leq \frac{2}{T}\eps^3.$$

On $S(0)$, the conditions on $\eps$ further imply $\eta \leq T$ and the useful inequalities
\begin{equation}
\label{rw2}
\left \{
\begin{gathered}
T\eps^{-2} < a_0 - \delta_r < r  < a_0 + \delta_r < 3T\eps^{-2}\\
-\frac{3}{2}\eps^{-2} < a_1 - \delta_w < w < a_1 + \delta_w < -\frac{1}{2} \eps^{-2}.
\end{gathered}
\right.
\end{equation}
Additionally, \eqref{rw2} combined with \eqref{suppcondl} implies a uniform upper bound on the angular momentum on $S(0)$, namely
\begin{equation}
\label{l2}
\ell < \left (\frac{3T\eps^{-2}}{(T + \eta) \eps^{-2}} \right )^2 \eps^2 \leq 9\eps^2 \leq 9.
\end{equation}

Now, since particle trajectories are convex, they must each attain a minimum, and we use this construction to create a uniform lower bound on the time until particles reach their minimum value. Because the enclosed mass satisfies $0 \leq m(t,r) \leq C_1$ for all $t, r \geq 0$, we see that $\mcR(t)$ satisfies
$$0 \leq \ddot{\mcR}(t) - \ell\mcR(t)^{-3} \leq C_1 \mcR(t)^{-2}$$
with $\mcR(0) = r > 0$ and $\dot{\mcR}(0) = w < 0$. 
As in the proof of Theorem \ref{T1}, we must exclude those particles in $S(0)$ with vanishing angular momentum, and thus we again let  $$S_+ = \{ (r, w, \ell) \in S(0): \ell > 0\}.$$
Using Lemma \ref{L1} with $L = \ell$ and $P = C_1$, we find for each $(r,w,\ell) \in S_+$ a time $T_0(r,w,\ell)$ such that
$$T_0 \geq \frac{r}{\vert w \vert} \left ( 1 - \sqrt{\frac{\ell+C_1r}{r^2w^2 + \ell + C_1r}} \right )$$
and
$$\dot{\mcR}(t) \leq 0 \quad \mathrm{for} \quad  t \in [0,T_0].$$
Next, we use \eqref{rw2}, \eqref{l2}, and the basic inequality
$$ \frac{1}{A + x} \geq \frac{1}{A} - \frac{x}{A^2}$$
for any $x,A > 0$ in order to find
\begin{eqnarray*}
T_0 & > & \frac{r}{\vert w \vert} \left ( 1 - \frac{\sqrt{\ell + C_1r}}{r \vert w \vert} \right )\\
& = & \frac{r}{\vert w \vert} - \frac{\sqrt{\ell + C_1r}}{w^2}\\
& \geq & \frac{a_0 -\delta_r}{\vert a_1 \vert + \delta_w} - \frac{\sqrt{9 + 3C_1T\eps^{-2}}}{\left ( \frac{1}{2} \eps^{-2} \right )^2}\\
& \geq & (a_0 -\delta_r)\left (\frac{1}{\vert a_1 \vert} -  \frac{\delta_w}{a_1^2} \right) - 4\eps^4\sqrt{9 + 3C_1T\eps^{-2}}\\
& \geq & \frac{a_0}{\vert a_1 \vert} - \frac{\delta_r}{\vert a_1 \vert} - \frac{a_0 \delta_w}{a_1^2} - 4\eps^3\sqrt{9\eps^2 + 3C_1T}\\
& \geq & \frac{a_0}{\vert a_1 \vert} - \left (\eps^5 + 2\eps^{5} + 4\eps^3\sqrt{9 + 3C_1T} \right )\\
& \geq & \frac{a_0}{\vert a_1 \vert} - \eta = T.
\end{eqnarray*}

Therefore, $T \in [0, T_0)$ for every $(r,w,\ell) \in S_+$ and we apply Lemma \ref{L2} to find
$$\mcR(T)^2 \leq (r + wT)^2 + (\ell r^{-2} + C_1r^{-1})T^2.$$
Because $T = \frac{a_0}{|a_1|} - \eta$ this further implies
\begin{eqnarray*}
\mcR(T)^2 & \leq & \left (r + \frac{a_0}{|a_1|}w - \eta w \right)^2 + \ell r^{-2} \left ( \frac{a_0}{|a_1|} - \eta \right)^2 + C_1r^{-1}T^2\\
& = & \left (r + \frac{a_0}{|a_1|}w \right)^2  +  \ell r^{-2}\left (\frac{a_0}{a_1} \right)^2\\
& \ & +2|w|\eta \left \vert r + \frac{a_0}{|a_1|}w \right \vert + w^2\eta^2 +  \ell r^{-2}\eta^2 + C_1r^{-1}T^2.
\end{eqnarray*}

Using \eqref{suppcond}, \eqref{suppcondrw}, \eqref{rw2}, and \eqref{l2} this yields
\begin{eqnarray*}
\mcR(T)^2 & \leq & \frac{\eps^2}{a_1^2} + 2\left ( \frac{3}{2} \eps^{-2} \right )C_0\eps^3 \frac{\eps}{\vert a_1 \vert} + \left ( \frac{3}{2} \eps^{-2} \right )^2(C_0\eps^3)^2 \\
& \ &   + 9\left (T\eps^{-2} \right )^{-2}(C_0\eps^3)^2 + C_1\left (T\eps^{-2} \right )^{-1}T^2\\
& \leq & \eps^6 + 3C_0 \eps^4 + 3C_0^2\eps^2 + \frac{9C_0^2}{T^2}\eps^{10} + C_1T\eps^2\\
& \leq & 64C_0^2\eps^2.
\end{eqnarray*}
Since this provides a uniform bound on $\mcR(T)$ over the set $S_+$, we take the supremum over all such triples to find
$$\sup_{(r,w,\ell) \in S(0)} \mcR(T, 0, r, w,\ell) = \sup_{(r,w,\ell) \in S_+} \mcR(T, 0, r, w,\ell) \leq 8C_0 \eps.$$


Finally, using Lemma \ref{L3}, the upper bound on spatial characteristics implies a lower bound on the density and therefore,
$$ \Vert \rho(T) \Vert_\infty \geq \frac{3C_1}{\left ( 8C_0\eps \right)^3 } \geq \frac{3C_1}{(8C_0)^3\eps^2 } \geq C_2.$$
The same lemma also provides a lower bound on the field so that 
$$\Vert E(T) \Vert_\infty \geq \frac{C_1}{\left ( 8C_0\eps \right)^2 } \geq 48C_0C_2 \geq C_2,$$
and this completes the proof.
\end{proof}

\section{Lemmas}\label{lemmas}
The first lemma uses the convex nature of particle characteristics to estimate their minimal value and the corresponding time at which it is achieved.

\begin{lemma}
\label{L1}
Let $L > 0$, $P \geq 0$, $y_0 > 0$ and $y_1 < 0$ be given.
Assume $y \in C^2\left ([0,\infty);(0,\infty)\right )$ satisfies
$$ 0 \leq \ddot{y}(t) - Ly(t)^{-3} \leq Py(t)^{-2}$$
for all $t > 0$ with $y(0) = y_0$ and $\dot{y}(0) = y_1$.
Then, we have the following: 
\begin{enumerate}
\item There exists a unique $T_0 > 0$ such that
$$\dot{y}(t)  < 0 \ \mathrm{for} \ t \in [0,T_0),$$ 
$$\dot{y}(T_0)  =0, \ \mathrm{and}$$
$$\dot{y}(t)  > 0 \ \mathrm{for} \ t \in (T_0,\infty).$$

\item Furthermore, define $$y_* = y_0\sqrt{\frac{L + Py_0}{y_0^2y_1^2 + L + Py_0}}.$$
Then, 
$$y(T_0) \leq y_* \quad \mathrm{and} \quad T_0 \geq  \frac{y_0 - y_*}{\vert y_1 \vert}.$$
\end{enumerate}
\end{lemma}

\begin{proof}
To begin, define
$$T_0 = \sup \{t \geq 0 : \dot{y}(t) \leq 0 \}$$
and note that $y_1 < 0$ implies $T_0 > 0$.
We first show that $T_0 < \infty$. For the sake of contradiction, assume $T_0 =\infty$. Then, we have $\dot{y}(t) \leq 0$ for all $t \geq 0$ and thus $y(t) \leq y_0$ for all $t \geq 0$.  From the lower bound on $\ddot{y}$, we find
$$ \ddot{y}(t) \geq Ly(t)^{-3} \geq Ly_0^{-3}$$
and hence for all $t \geq 0$
$$\dot{y}(t) \geq Ly_0^{-3}t + y_1.$$
Taking $t > \frac{-y_1 y_0^3}{L}$ implies $\dot{y}(t) > 0$, thus contradicting the assumption that $T_0 = \infty$, and we conclude that $T_0$ must be finite.

Next, for $t \in [0,T_0]$, we multiply the differential inequality
$$\ddot{y}(t) - Ly(t)^{-3} \leq Py(t)^{-2}$$
by $-\dot{y}(t)$ and integrate over $[0,t]$ to find
\begin{equation}
\label{partenergy}
\dot{y}(t)^2 + Ly(t)^{-2} + 2Py(t)^{-1} \geq y_1^2 + Ly_0^{-2} + 2Py_0^{-1}.
\end{equation}
Using the decreasing nature of $y$ on this interval, so that $y(t) \leq y_0$, we find
$$ y_0^{-2} - y(t)^{-2} \leq 0 \quad \mathrm{and} \quad y_0^{-1} + y(t)^{-1} \geq 2 y_0^{-1}$$
and within \eqref{partenergy} this implies
\begin{eqnarray*}
\dot{y}(t)^2 & \geq & y_1^2 + L(y_0^{-2} - y(t)^{-2}) +2 P(y_0^{-1}- y(t)^{-1})\\
& = & y_1^2 + \left (L + \frac{2P}{y_0^{-1}+y(t)^{-1}} \right) (y_0^{-2} - y(t)^{-2})\\
& \geq & y_1^2 + (L + Py_0) (y_0^{-2} - y(t)^{-2}).
\end{eqnarray*}
Thus, we have for all $t \in [0,T_0]$
\begin{equation}
\label{upper1}
\dot{y}(t)^2 \geq y_1^2 + (L + Py_0) (y_0^{-2} - y(t)^{-2}).
\end{equation}
Evaluating this inequality at $t = T_0$ yields
$$y_1^2 + (L + Py_0) (y_0^{-2} - y(T_0)^{-2}) \leq 0$$
and rearranging gives $y(T_0) \leq y_*$ with
$$y_* = y_0\sqrt{\frac{L + Py_0}{y_0^2y_1^2 + L + Py_0}}.$$

Next, the lower bound in the differential inequality implies
$$\ddot{y}(t) \geq Ly(t)^{-3} \geq 0,$$
and thus $\dot{y}(t) \geq y_1$
for all $t \in [0,T_0].$
Integrating over $[0,T_0]$ produces
$y(T_0) - y_0 \geq y_1 T_0$
and since $y_1 < 0$, we find
$$ T_0 \geq \frac{1}{y_1} (y(T_0) - y_0) \geq \frac{y_* - y_0}{y_1} = \frac{y_0 - y_*}{\vert y_1 \vert}.$$
Finally, the convexity of $y(t)$ implies the uniqueness of $T_0$ and the proof is complete.
\end{proof}

Next, we state and prove a result that provides an upper bound on particle positions over the interval of time on which they remain radially decreasing. This bound allows us to relate particle trajectories at any time to their starting positions and momenta, as well as their angular momentum and the total mass.

\begin{lemma}
\label{L2}
Let $y(t)$ and $T_0 > 0$ satisfy the conditions of Lemma \ref{L1}.  Then, for all $t \in [0,T_0]$, we have
$$y(t)^2 \leq (y_0 + y_1t)^2 + (Ly_0^{-2} + Py_0^{-1} )t^2.$$
\end{lemma}

\begin{proof}
As in the proof of Lemma \ref{L1}, we return to \eqref{upper1} and multiply by $y(t)^2$ to find
\begin{equation}
\label{y2energy}
\left [ \frac{1}{2} \frac{d}{dt} ( y(t)^2 ) \right ]^2 \geq y_0^{-2}(y_0^2y_1^2 + L+ Py_0)y(t)^2 - (L + Py_0).
\end{equation}
Now, if 
\begin{equation}
\label{lb}
y(t)^2 > y_0^2 \frac{L + Py_0}{y_0^2y_1^2 + L + Py_0}
\end{equation}
then the right side of \eqref{y2energy} is positive and we find
$$\frac{1}{2} \left \vert \frac{d}{dt} ( y(t)^2 ) \right \vert \geq \sqrt{y_0^{-2}(y_0^2y_1^2 + L+ Py_0)y(t)^2 - (L + Py_0)},$$
which can be rewritten as
$$\frac{\frac{1}{2} \frac{d}{dt} ( y(t)^2 )}{ \sqrt{y_0^{-2}(y_0^2y_1^2 + L+ Py_0)y(t)^2 - (L + Py_0)}} \leq -1.$$
Integrating yields
$$\sqrt{y_0^{-2}(y_0^2y_1^2 + L+ Py_0)y(t)^2 - (L + Py_0)} - \vert y_1 \vert y_0 \leq - y_0^{-2} (y_0^2y_1^2 + L + Py_0) t$$
so that
$$y(t)^2 \leq y_0^2(y_0^2y_1^2 + L+ Py_0)^{-1} \left [L + Py_0 + (y_1y_0 + y_0^{-2} (y_0^2y_1^2 + L+ Py_0)t )^2\right ]$$
and after some algebra this becomes
$$y(t)^2 \leq (y_0 + y_1t)^2 + (Ly_0^{-2} + Py_0^{-1})t^2.$$
Including the assumption \eqref{lb}, this implies
$$y(t)^2 \leq \max \left \{y_0^2 \frac{L + Py_0}{y_0^2y_1^2 + L + Py_0},  (y_0 + y_1t)^2 + (Ly_0^{-2} + Py_0^{-1})t^2 \right \}.$$
Finally, the upper bound in this estimate from condition \eqref{lb} can be removed by noting that $y_0^2 \frac{L + Py_0}{y_0^2y_1^2 + L + Py_0}$ is, in fact, the minimum of the parabola in $t$, which occurs at the time $$t_{min} = \frac{-y_1y_0^3}{y_0^2y_1^2 + L + Py_0}$$ and the conclusion follows.
\end{proof}

Our final lemma provides lower bounds on the charge density and electric field in terms of the total mass and the position of the particle on the support of $f_0$ that is furthest from the origin.

\begin{lemma}
\label{L3}
Let $f(t,r, w, \ell)$ be a spherically-symmetric solution of \eqref{VP} with associated charge density $\rho(t,r)$ and electric field $E(t,x)$, and let $\left (\mcR(t, 0, r, w, \ell), \mcW(t, 0, r, w, \ell), \mcL(t, 0, r, w, \ell) \right)$ be a characteristic solution of \eqref{charang}.  
If at some $T \geq 0$ we have
$$ \sup_{(r, w, l) \in S(0)} \mcR(T, 0, r, w, \ell) \leq B,$$ then
$$\Vert \rho(T) \Vert_\infty \geq \frac{3M}{4\pi B^3}$$
and
$$\Vert E(T) \Vert_\infty \geq \frac{M}{B^2}.$$
\end{lemma}

\begin{proof}
Let $f(t,r,w,\ell)$ be a given spherically-symmetric solution with initial data $f_0(r,w,\ell)$. As previously mentioned, one may integrate the Vlasov equation \eqref{vlasovang} over phase space to find that the total mass of the system in conserved in these coordinates, namely
$$ M= 4\pi^2 \iiint_{S(0)} f_0(r,w,\ell) \ d\ell dw dr = 4\pi^2 \iiint_{S(t)} f(t,r,w,\ell) \ d\ell dw dr$$
for every $t \geq 0$.
Due to the bound on spatial characteristics, it follows that 
$S(T) \subset [0,B] \times \mathbb{R} \times [0,\infty)$. 
Hence, using the radial form of the density in \eqref{rhoang}, we find
\begin{eqnarray*}
M & = & 4\pi^2 \iiint_{S(T)} f(T,r,w,\ell) \ d\ell dw dr.\\
& \leq & 4\pi^2 \int_0^B \int_{-\infty}^\infty \int_0^\infty f(T,r,w,\ell) \ d\ell dw dr.\\
& = & 4\pi \int_0^B r^2 \rho(T,r) \ dr\\
& \leq & 4\pi \Vert \rho(T) \Vert_\infty \left (\int_0^B r^2 \ dr \right ).
\end{eqnarray*}
Integrating and rearranging the inequality yields
$$\Vert \rho(T) \Vert_\infty \geq \frac{3 M}{4\pi B^3}.$$

To prove the second conclusion, we use a similar argument. Let $$\mfR(T) =\sup_{(r, w, \ell) \in S(0)} \mcR(T, 0, r, w, \ell).$$ From the Vlasov equation, we have for every $s \geq 0$
$$ f(s,\mcR(s, 0, r, w, \ell), \mcW(s, 0, r, w, \ell), \mcL(s, 0, r, w, \ell)) = f_0(r, w, \ell).$$
Then, 
using the change of variables
$$ \left \{
\begin{gathered}
r = \mcR(0,T, \tilde{r}, \tilde{w}, \tilde{\ell})\\
w = \mcW(0,T, \tilde{r}, \tilde{w}, \tilde{\ell})\\
\ell = \mcL(0,T, \tilde{r}, \tilde{w}, \tilde{\ell})
\end{gathered}
\right. $$ 
along with the inverse mapping
$$ \left \{
\begin{gathered}
\tilde{r} = \mcR(T, 0, r, w, \ell)\\
\tilde{w} = \mcW(T, 0, r, w, \ell)\\
\tilde{\ell} = \mcL(T, 0, r, w, \ell)
\end{gathered}
\right. $$ 
and the well-known measure-preserving property (cf. \cite{Glassey}) which guarantees
$$\left \vert \frac{\partial(r, w, \ell)}{\partial (\tilde{r}, \tilde{w}, \tilde{\ell})} \right \vert = 1,$$
it follows that
\begin{equation}
\label{mpp}
\int_0^{\mfR(T)} \int_0^\infty \int_{-\infty}^\infty f(T, \tilde{r}, \tilde{w}, \tilde{\ell}) \ d\tilde{w} d\tilde{\ell} d\tilde{r} = \int_0^{\mfR(0)} \int_0^\infty \int_{-\infty}^\infty f_0(r, w, \ell) \ dw d\ell dr.
\end{equation}
Now, from \eqref{fieldang} we find
$$\vert E(T, \mfR(T))\vert = \frac{m(T, \mfR(T))}{\mfR(T)^2}.$$
Inserting \eqref{rhoang} into \eqref{massang} and using \eqref{mpp}, we have
\begin{eqnarray*}
m(T, \mfR(T)) & = & 4\pi^2 \int_0^{\mfR(T)} \int_0^\infty \int_{-\infty}^\infty f(T, \tilde{r}, \tilde{w}, \tilde{\ell}) \ d\tilde{w} d\tilde{\ell} d\tilde{r}\\
& = & 4\pi^2 \int_0^{\mfR(0)} \int_0^\infty \int_{-\infty}^\infty f_0(r, w, \ell) \ dw d\ell dr\\
& = & 4\pi^2 \iiint_{S(0)} f_0(r, w, \ell) \ dw d\ell dr\\
& = & M.
\end{eqnarray*}
Therefore, using the condition on the spatial characteristics
$$\vert E(T, \mfR(T))\vert = \frac{M}{\mfR(T)^2} \geq \frac{M}{B^2}.$$
Finally, since $E(T,r)$ obtains this value at some $r>0$, we have
$$\Vert E(T) \Vert_\infty \geq \frac{M}{B^2}.$$
and the proof is complete.
\end{proof}

%
%

\end{document}